\documentclass[12pt, twoside]{article}
\usepackage{amsmath,amsthm,amssymb}
\usepackage{times}
\usepackage{enumerate}

\def\titlerunning#1{\gdef\titrun{#1}}
\makeatletter
\def\author#1{\gdef\autrun{\def\and{\unskip, }#1}\gdef\@author{#1}}
\def\address#1{{\def\and{\\\hspace*{18pt}}\renewcommand{\thefootnote}{}%
\footnote {#1}}%
\markboth{\autrun}{\titrun}}
\makeatother
\def\email#1{e-mail: #1}
\def\subjclass#1{{\renewcommand{\thefootnote}{}%
\footnote{\emph{Mathematics Subject Classification (2010):} #1}}}
\def\keywords#1{\par\medskip
\noindent\textbf{Keywords.} #1}

\newtheorem{thm}{Theorem}[section]
\newtheorem{cor}[thm]{Corollary}
\newtheorem{lem}[thm]{Lemma}

\theoremstyle{definition}
\newtheorem{defin}[thm]{Definition}
\newtheorem{rem}[thm]{Remark}

\numberwithin{equation}{section}

\def\R {\Bbb R}
\def\N {{\Bbb N}}

\newtheorem{pro}[thm]{Proposition}
\newtheorem{ques}[thm]{Question}

\begin{document}


\baselineskip=17pt


\titlerunning{Polynomials with bounded integer coefficients}

\title{On the topology of polynomials with bounded integer coefficients}

\author{De-Jun Feng
}

\date{}

\maketitle

\address{D. J. Feng: Department of Mathematics,
The Chinese University of Hong Kong,
Shatin,  Hong Kong; 
 \email{djfeng@math.cuhk.edu.hk}}



\begin{abstract}

For a  real number $q>1$ and a positive integer $m$, let
 $$
 Y_m(q):=\left\{\sum_{i=0}^n\epsilon_i q^i:\; \epsilon_i\in \{0, \pm 1,\ldots, \pm m\},\; n=0, 1,\ldots \right\}.
 $$
In this paper, we show that $Y_m(q)$ is dense in $\R$ if and only if $q<m+1$ and $q$ is not a Pisot number. This completes several previous results and answers an open question raised by
Erd\H{o}s, Jo\'{o} and  Komornik  \cite{EJK98}.

\end{abstract}

\section{Introduction}

 For a  real number $q>1$ and a positive integer $m$, let
 $$
 Y_m(q):=\left\{\sum_{i=0}^n\epsilon_i q^i:\; \epsilon_i\in \{0, \pm 1,\ldots, \pm m\},\; n=0, 1,\ldots \right\}.
 $$
In this paper, we consider the following old question regarding of the topological structure of $Y_m(q)$:

\begin{ques}
\label{ques-1.1}
 For which pair $(q, m)$ is the set  $Y_m(q)$  dense in $\R$?
\end{ques}

  It is well known that $Y_m(q)$ is not dense in $\R$ in the following two cases:  $q$ is a Pisot number (Garsia \cite{Gar62}), or  $q\geq m+1$ (Erd\H{o}s and Komornik \cite{EK98}); Recall that a {\it Pisot number} is  an algebraic integer $>1 $ all of whose conjugates have modulus $<1$ (cf. \cite{Sal63}).  For the reader's convenience, we include a brief proof. First assume that $q$ is  a Pisot number. Denote by $q_1,\ldots, q_d$ the algebraic conjugates of $q$. Then $\rho:=\max_{1\leq j\leq d}|q_j|<1$.  Let $P(x)=\sum_{i=0}^n\epsilon_i x^i$ be a polynomial with coefficients in $\{0, \pm 1,\ldots, \pm m\}$. Suppose that $P(q)\neq 0$. Then $P(q_j)\neq 0$ for $1\leq j\leq d$.
Hence $P(q)\prod_{j=1}^dP(q_j)$ is a non-zero integer. Therefore $$
|P(q)|\geq \prod_{j=1}^d \frac{1}{|P(q_j)|}\geq \left(\frac{1}{\sum_{i=0}^{n}m\rho^i}\right)^d>m^{-d}(1-\rho)^d.
$$   
It follows that  $Y_m(q)$ is not dense in $\R$ with $0$ being an isolated point. The same argument also shows that $0$ is an isolated point of $Y_{2m}(q)=Y_m(q)-Y_m(q)$, therefore  $Y_m(q)$ is uniformly discrete  in $\R$. Next assume that $q\geq m+1$. Then for any $n\in \N$,
$$
q^n-\sum_{i=0}^{n-1}mq^i=\frac{q^n(q-1-m)+m}{q-1}\geq \frac{(q-1-m)+m}{q-1}=1. 
$$ 
It follows that $|P(q)|\geq 1$ for any polynomial $P$ with degree $\geq 1$ and coefficients in $\{0, \pm 1,\ldots, \pm m\}$. Hence $Y_m(q)\cap (-1, 1)=\{0\}$, as a consequence, $Y_m(q)$ is not dense in $\R$.

In this paper, by proving the reverse direction we obtain the following theorem, which provides a complete answer to Question  \ref{ques-1.1}.

\begin{thm}
\label{thm-1.1}
$Y_m(q)$ is dense in $\R$ if and only if $q< m+1$ and $q$ is not a Pisot number.
\end{thm}

 We remark that Question \ref{ques-1.1} is closely related to a  project proposed by Erd\H{o}s, Jo\'{o} and Komornik in last 90's.  For  $q>1$ and  $m\in \N$, let
$$X_m(q)=\left\{\sum_{i=0}^n\epsilon_i q^i:\; \epsilon_i\in \{0, 1,\ldots,m\},\; n=0, 1,\ldots \right\}.$$
 Since $X_m(q)$ is discrete, we may arrange the points of $X_m(q)$ into an increasing sequence:
$$0=x_0(q, m)<x_1(q,m)<x_2(q,m)<\cdots.$$
Denote 
\begin{equation*}
\displaystyle\ell_m(q)=\liminf_{n\to \infty}(x_{n+1}(q,m)-x_n(q,m)), \quad L_m(q)=\limsup_{n\to \infty} (x_{n+1}(q,m)-x_n(q,m)).\end{equation*}
Originated from the study of expansions in non-integer bases,  Erd\H{o}s, Jo\'{o} and Komornik  \cite{EJK90, EJK98, EK98} proposed to characterize all the pairs  $(q,m)$ so that $\ell_m(q)$ and $L_m(q)$ vanish. By definition, $\ell_m(q)=0$ is equivalent to that $0$ is an accumulation point of $Y_m(q)$. However, it was proved by Drobot \cite{Dro73} (see also \cite{DrMc80})  that  $Y_m(q)$ is dense in $\R$ if and only if $0$ is an accumulation point of  $Y_m(q)$.  Hence  $\ell_m(q)=0$ if and only if $Y_m(q)$ is dense in $\R$.    In \cite{EJK98}, Erd\H{o}s, Jo\'{o} and Komornik raised the open question whether
   $\ell_1(q)=0$ for any non-Pisot number $q\in (1,2)$. This question was also formulated  in \cite{SiSo11, AkKo11}.   As a direct corollary of Theorem \ref{thm-1.1} and Drobot's result, we can provide a confirmative answer to this question.

\begin{cor}
\label{cor-1.2}
$\ell_m(q)=0$ if and only if $q< m+1$ and $q$ is not a Pisot number.
\end{cor}

In the literature there are  some  partial results on  Question \ref{ques-1.1} and the project of  Erd\H{o}s et al. It was shown in  \cite{Dro73, DrMc80} that if $q\in (1,m+1)$ does not satisfy an algebraic equation with coefficients $0, \pm 1,\ldots, \pm m$, then $\ell_m(q)=0$.
In \cite{Bug96} Bugeaud showed that if $q$ is not a Pisot number, then there exists an integer $m$ so that $\ell_m(q)=0$. The approach of Bugeaud did not provide any estimate of $m$. A substantial progress was made later by Erd\H{o}s and Komornik \cite{EK98}, who proved that   $\ell_m(q)=0$ if $q$ is not a Pisot number and $m\geq \lceil q-q^{-1} \rceil+\lceil q-1\rceil$, where $\lceil x \rceil$ denotes the smallest integer $\geq x$.  Recently Akiyama and Komornik \cite{AkKo11}   showed that $\ell_1(q)=0$ if $q\in (1,  \sqrt{2}]$ is not a Pisot number smaller than the golden ratio $(1+\sqrt{5})/2$.
In \cite{SiSo11}, Sidorov and Solomyak proved that if $q\in (1,m+1)$ and $q$ is not a Perron number, then  $\ell_m(q)=0$.  Recall that an algebraic integer $q>1$ is called a {\it Perron number} if each of its conjugates is less than $q$ in modulus.

As for the value of $L_m(q)$, Erd\H{o}s and Komornik \cite{EK98} proved that   $L_m(q)>0$ if $q$ is a Pisot number or $q\geq (m + \sqrt{m^2+4})/2$.  It remains an open problem whether $L_m(q)=0$ for all other pairs $(q,m)$ with $q>1$ and $m\in \N$. In \cite{Kom02}, Komornik  conjectured that this is true in the case when $m=1$, i.e., $L_1(q)=0$ for any non-Pisot number smaller than the golden ratio.  Some partial results  were obtained by Erd\H{o}s-Komornik  and Akiyama-Komornik:  $L_m(q)=0$ if $q$ is non-Pisot and $m\geq \lceil q-q^{-1} \rceil+2\lceil q-1\rceil$ (\cite{EK98}); furthermore, $L_1(q)=0$ if $1<q\leq \sqrt[3]{2}\approx 1.2599$ (\cite{EK98,AkKo11}). Here the second part was only proved in \cite{EK98} for all $1<q\leq \sqrt[4]{2}\approx1.1892$ with the possible exception of the square root of the second Pisot number.

By directly applying Corollary \ref{cor-1.2} and \cite[Lemma 2.5]{AkKo11} (which says that $\ell_m(q^2)=0$ implies \footnote
{This implication was first proved in \cite[Theorem 5]{EJK98}  in the case $m=1$. It extends to $m>1$ directly. } $L_m(q)=0$), we have the following theorem which improves the results in \cite{EK98,AkKo11}.
\begin{thm}
\label{thm-1.3'}
If $1<q< \sqrt{m+1}$ and $q^2$ is not a Pisot number, then $L_m(q)=0$. In particular, if $q\in (1, \sqrt{2})$ and $q^2$ is not a Pisot number, then $L_1(q)=0$.
\end{thm}

  Let us mention some other important results related to Question \ref{ques-1.1}. In \cite{EK98}, Erd\H{o}s and Komornik  showed that  if $q>1$ is not a Pisot number and $m\geq q-q^{-1}$, then $Y_m(q)$ has a finite accumulation point. Very recently,  Akiyama and Komornik \cite{AkKo11} characterized all pairs $(q,m)$ so that $Y_m(q)$ has a finite accumulation point, completing the previous results of Erd\H{o}s and Komornik \cite{EK98} and Zaimi \cite{Zai07} on this topic.

\begin{thm}[Akiyama and Komornik \cite{AkKo11}] $Y_m(q)$ has a finite accumulation point in $\R$ if and only if $q<m+1$ and $q$ is not a Pisot number.
 \label{thm-1.2}
 \end{thm}

 In this paper, we shall prove the following result.
 \begin{thm}
 \label{thm-1.3} Assume that $1<q\leq m+1$. Then $Y_m(q)$ has no  finite accumulation points in $\R$  if and only if  $0$ is not an accumulation point of $Y_m(q)$.
 \end{thm}

Theorem \ref{thm-1.3}  was conjectured at the end of \cite{AkKo11},  where the authors
observed that, combined with Theorem \ref{thm-1.2}, this would yield that $0$ is an accumulation point of $Y_m(q)$ (equivalently, $Y_m(q)$ is dense in $\R$) if and only if $q< m+1$ and $q$ is not a Pisot number. Hence Theorem \ref{thm-1.1} follows from Theorem \ref{thm-1.3} and Theorem \ref{thm-1.2}.

We remark that the separation property of $Y_m(q)$ were also  considered by Lau \cite{Lau93} in his study of Bernoulli convolutions (see \cite{PSS} for a survey about Bernoulli convolutions). Following Lau \cite{Lau93}, we call  $q\in (1,2)$  a {\it F-number} if $$Y_1(q)\cap \Big[-\frac{1}{q-1}, \frac{1}{q-1}\Big] \; \mbox{ is a finite set. } $$ Clearly, each Pisot number in $(1,2)$ is a F-number.   Lau raised a question in \cite{Lau93} whether or not there exists a F-number which is non-Pisot. As a corollary of Theorem \ref{thm-1.1} (it also follows from Theorem \ref{thm-1.2} together with  Remark \ref{re-1.1} and Lemma \ref{lem-2.1}), we have the following answer to Lau's question.
\begin{cor}
Every F-number is a Pisot number.
\end{cor}

As a closely related topic, for $q\in (1,2)$, the topological structure of the following set
$$
A(q)=\left\{\sum_{i=0}^n\epsilon_i q^i:\; \epsilon_i\in \{-1, 1\},\; n=0, 1,\ldots \right\}
$$
has been studied in the literature \cite{PeSo00, BoHa02, Sta10, AkKo11}. It was proved  that if $1<q\leq \sqrt{2}$ is not a Pisot number, then  $A(q)$ is dense in $\R$  \cite{AkKo11};  moreover, for almost all $q\in (\sqrt{2}, 2)$, $A(q)$ is dense in $\R$ \cite{PeSo00}.  Meanwhile,  there exist non-Pisot numbers $q\in (\sqrt{2}, 2)$ such that $A(q)$ is discrete \cite{BoHa02}. It is an interesting question to characterize all
$q\in (\sqrt{2}, 2)$ so that $A(q)$ is dense in $\R$.

 The proof of Theorem \ref{thm-1.3} is based on our study on separation properties of
  homogeneous iterated function systems  (IFS) on $\R$.
Let $m$ be a positive integer  and $\Phi=\{\phi_i\}_{i=0}^{m}$  a family of contractive maps on $\R$ of the form:
$$
\phi_i(x)=\rho x+b_i,\quad i=0,1,\ldots, m,
$$
where
\begin{equation}
\label{e-1'}
0<\rho<1 \quad \mbox{ and }\quad  0=b_0<\ldots<b_m=1-\rho.
\end{equation}
$\Phi$ is called a {\it homogeneous iterated function system} on $\R$.
According to Hutchinson \cite{Hut81}, there is a unique compact set $K:=K_\Phi\subset \R$  such that
$$
K=\bigcup_{i=0}^m \phi_i(K).
$$
We call $K$  the {\it attractor} of $\Phi$.   It is easy to check that
$$
K=\left\{\sum_{n=0}^\infty b_{i_n}\rho^n:\; i_n\in \{0,1,\ldots,m\} \mbox{ for }n\geq 0\right\}.
$$
The condition \eqref{e-1'} implies that the convex hull of $K$ is the unit interval $[0,1]$.

For any finite word $I=i_1i_2\ldots i_n\in \{0,1,\ldots,m\}^n$, write
$\phi_I=\phi_{i_1}\circ\phi_{i_2}\circ\ldots\circ \phi_{i_n}$.
Clearly, $\phi_I(0)=b_{i_1}+\rho b_{i_2}+\ldots+\rho^{n-1} b_{i_n}.$
\begin{defin}
\label{de-1}
Say that $\Phi$ satisfies the \emph{weak separation condition} if there exists a constant $c>0$ such that for any $n\in \N$ and any $I,J\in \{0,1,\ldots, m\}^n$,
\begin{equation*}
\mbox{either}\qquad \rho^{-n}|\phi_I(0)-\phi_{J}(0)|=0\qquad
\mbox{or} \qquad \rho^{-n}|\phi_I(0)-\phi_{J}(0)|\geq c.
\end{equation*}
\end{defin}

\begin{defin}
\label{de-2}
Say that $\Phi$ satisfies the \emph{finite type condition} if there exists a finite set $\Gamma\subset [0,1)$   such that for any $n\in \N$ and any $I,J\in \{0,1,\ldots, m\}^n$,
\begin{equation*}
\mbox{either}\qquad \rho^{-n}|\phi_I(0)-\phi_{J}(0)|\geq 1\qquad
\mbox{or} \qquad \rho^{-n}|\phi_I(0)-\phi_{J}(0)|\in \Gamma.
\end{equation*}
\end{defin}

\begin{rem}
\label{re-1.1}
 By definition,  for $1<q<2$, $q$ is a F-number if and only if
the IFS $\{q^{-1}x, q^{-1}x+(1-q^{-1})\}$ satisfies the finite type condition.
\end{rem}

The concepts of weak separation condition and  finite type condition were respectively introduced  in \cite{LaNg99, NgWa01} in  more general settings for the study of IFSs with overlaps.  One is referred to \cite{Zer96, DLN11} for some equivalent definitions.

It is easy to see that in our setting, the finite type condition implies the  weak separation condition (this is also true in the general settings of \cite{LaNg99, NgWa01}; see \cite{NN02} for a proof). However it is not clear whether the  weak separation condition also implies the finite type condition in our setting. The following theorem gives this implication under an additional assumption on $\Phi$.

\begin{thm}
\label{thm-2.1} Let $\Phi=\{\phi_i(x)=\rho x+b_i\}_{i=0}^m$ be an IFS satisfying \eqref{e-1'}.  Assume in addition that
\begin{equation}
\label{e-2.1}
b_{i+1}-b_i\leq \rho \quad \mbox{ for all}\quad 0\leq i\leq m-1.
\end{equation}
Suppose $\Phi$ satisfies the weak separation condition. Then $\Phi$ also satisfies the finite type condition.
\end{thm}

We remark that the condition \eqref{e-2.1} is equivalent to $[0,1]=\bigcup_{i=0}^m \phi_i([0,1])$, i.e., $K_\Phi=[0,1]$.

Now for a given pair $(q,m)$ with $1<q\leq m+1$, consider a special IFS  $\Phi=\{\rho x+b_i\}_{i=0}^m$ with $\rho={q}^{-1}$ and $b_i=i(1-{q}^{-1})/m$ for $0\leq i\leq m$.   Then $\Phi$ satisfies the assumptions in Theorem \ref{thm-2.1}.  However   $\Phi$ satisfies the weak separation condition if and only if $0$ is not an accumulation point of $Y_m(q)$; whilst $\Phi$ satisfies the finite type condition if and only if $Y_m(q)$ has no finite accumulation points in $\R$ (see Lemma \ref{lem-2.1}). Hence according to Theorem \ref{thm-2.1}, the condition that $0$ is not an accumulation point of $Y_m(q)$ implies that $Y_m(q)$ has no finite accumulation points in $\R$; from which Theorem \ref{thm-1.3} follows. As a corollary of this and Theorem \ref{thm-1.1}, we have

\begin{cor}
\label{cor-2.1}
 For  a given pair $(q,m)$ with $1<q<m+1$, let $\Phi$ denote the IFS $\{\phi_i(x)=q^{-1}x +i(1-q^{-1})/m\}_{i=0}^m$ on $\R$.  Then
 $\Phi$ satisfies the weak separation condition (resp. the finite type condition) if and only if  $q$ is a Pisot number.
\end{cor}

The paper is organized as follows. In Section \ref{S-2}, we prove Theorem \ref{thm-2.1}. In Section \ref{S-3}, we give some final remarks and questions.

\bigskip
\section{Separation properties of IFSs and the proof of Theorem \ref{thm-2.1}}
\label{S-2}
Before giving the proof of Theorem \ref{thm-2.1}, we first present two lemmas.

\begin{lem}
\label{lem-2.1}
Let $\Phi=\{\phi_i(x)=\rho x+b_i\}_{i=0}^m$ be an IFS on $\R$ with $$0<\rho<1,\qquad 0=b_0< \ldots< b_m=1-\rho.$$
Denote
$$Y=\left\{\sum_{i=1}^n\epsilon_i \rho^{-i}:\; \epsilon_i\in \{b_s-b_t: 0\leq s,t\leq m\},\; n= 1,\ldots \right\}.$$
 Then $\Phi$ satisfies the weak separation condition if and only if $0$ is not an accumulation point of $Y$; whilst $\Phi$ satisfies the finite type condition if and only if $Y$ has no finite accumulation points in $\R$.
\end{lem}

\begin{proof}
For $n\geq 1$,  $I=i_1\ldots i_n, J=j_1\ldots j_n\in \{0,1,\ldots,m\}^n$, we
have
\begin{equation}
\label{e-e1}
\rho^{-n}(\phi_I(0)-\phi_J(0))=\sum_{s=1}^n (b_{i_{s}}-b_{j_s}) \rho^{-(n+1-s)}=\sum_{s=1}^n (b_{i_{n+1-s}}-b_{j_{n+1-s}}) \rho^{-s}.
\end{equation}
Hence by Definition \ref{de-1}, $\Phi$ satisfies the weak separation condition if and only if $0$ is not an accumulation point of $Y$.
In the following  we show that $Y$ has no finite accumulation points if and only if $\Phi$ satisfies the finite type condition.

By   \eqref{e-e1} and Definition \ref{de-2}, we see that $\Phi$ satisfies the finite type condition if and only if  $Y\cap [-1,1]$ contains only finitely many points.
It is direct to see that $Y$ has no finite accumulation points implies $Y\cap [-1,1]$ contains only finitely many points.
Hence to finish the proof, we only need to show that the finiteness assumption of  $Y\cap [-1,1]$  implies that $Y$ has no finite accumulation points.

From now on, we  assume that   $Y\cap [-1,1]$ contains only finitely many points. Set $A=Y\cap [-1,1]$ and  $B=\{b_{i}-b_{j}: 0\leq i, j\leq m\}$. Since
$A$ and $B$ are finite sets, we can pick  $u>1$ such that $(1,u)\cap  \rho^{-1}(A+B)=\emptyset$, where
$$\rho^{-1}(A+B):= \{\rho^{-1} (x+\epsilon):\; x\in A,\; \epsilon\in B\}.$$
Since $0\in A$, we have $(1,u)\cap \rho^{-1} B=\emptyset$.
We first claim that  $Y\cap (1,u)=\emptyset$. To see this, for any $y\in Y$, let $\deg(y)$ denote the smallest $n\in \N$ such that
$
y=\sum_{i=1}^n\epsilon_i \rho^{-i}
$
for some $\epsilon_1,\ldots,\epsilon_n\in B$.  Assume  on the contrary that $Y\cap (1,u)\neq \emptyset$. Define
$$N=\min\{\deg(y):\; y\in Y\cap (1,u)\}.$$ 
Then $N\in \N$. Pick $z\in Y\cap (1,u)$ so that $\deg(z)=N$.
 Since $(1,u)\cap \rho^{-1}B= \emptyset$, we have $z\not\in \rho^{-1}B$ and thus $N=\deg(z)\geq 2$.  Then there exist $\epsilon_1,\ldots,\epsilon_N\in B$ such that
$$z=\sum_{i=1}^N\epsilon_i \rho^{-i}.$$
Denote $w=\sum_{i=1}^{N-1}\epsilon_{i+1} \rho^{-i}$. Then $w\in Y$  and $z=\rho^{-1}w+\rho^{-1}\epsilon_1$.  Notice that $w\not\in  A$ (and hence $|w|>1$); for otherwise we have $z\in \rho^{-1}(A+B)$, contradicting $(1,u)\cap \rho^{-1}(A+B)=\emptyset$ and $z\in (1,u)$. On the other hand, we must have $|w|<z$; if not,
$$
|\rho^{-1}\epsilon_1|=|\rho^{-1}w-z|\geq \rho^{-1}|w|-z\geq (\rho^{-1}-1)z>\rho^{-1}-1= \rho^{-1}\max B,
$$
leading to a contraction. Therefore, we have $1<|w|<z<u$, and thus $|w|\in Y\cap (1,u)$. However, $\deg(|w|)\leq N-1<\deg(z)$, contradicting the minimality of $\deg(z)$. Therefore, we must have  $Y\cap (1,u)=\emptyset$.

Since $Y=-Y$, we also have $Y\cap (-u,-1)=\emptyset$.  Thus $Y\cap (-u,u)$ contains only finitely many points.
In the end, we show that $Y$ has no finite accumulation points. Assume on the contrary that $Y$ has a finite accumulation point, saying $v$.  We derive a contradiction as below.  Note that  $Y\cap (-u, u)$ contains only finitely many points. Hence we must have $|v|\geq u$. Note that for any $n\in \N$,
\begin{equation}
\label{e-t}
Y=\rho^{-n}Y+D_n,
\end{equation}
where $D_n:=\{\sum_{i=1}^n\epsilon_i \rho^{-i}:\; \epsilon_i\in B \mbox{ for all }i \}$. Take a large $n$  such that $\rho^{n}|v|+1<u$. By \eqref{e-t}, $Y$ has a finite accumulation point $w$ (it is possible that $w\notin Y$),  and $z\in D_n$ such that $v= \rho^{-n}w+z$. Then
$$|w|=|\rho^n(v-z)|\leq \rho^n |v|+\rho^n\sum_{i=1}^n (1-\rho) \rho^{-i}<\rho^n |v|+1<u.$$
This contradicts the fact that $Y$ has no accumulation points in $(-u,u)$.
\end{proof}

\begin{lem}
\label{lem-2.2}
Let $\Phi=\{\phi_i(x)=\rho x+b_i\}_{i=0}^m$ be an IFS satisfying $$0<\rho<1,\qquad 0=b_0< \ldots< b_m=1-\rho$$ and
$$
b_{i+1}-b_i\leq \rho \quad \mbox{ for all}\quad 0\leq i\leq m-1.
$$
Then the following properties hold:
\begin{enumerate}[\upshape(1)]
\item For any $n\in \N$, we have $[0,1]=\bigcup_{I\in \{0,1,\ldots,m\}^n}\phi_I([0,1])$;
\item For  $n, k\in \N$ and $J\in \{0,1,\ldots, m\}^n$, if $[c,d]$ is a subinterval of $\phi_J([0,1])$ with length $\geq \rho^{n+k}$,  then there exists
$J'\in \{0,1,\ldots, m\}^k$ such that $\phi_{JJ'}(0)\in [c,d]$.
\end{enumerate}
\end{lem}
\begin{proof}
It is direct to check that $[0,1]=\bigcup_{i=0}^m \phi_i([0,1])$.  Iterating this relation for $n$ times yields (1).

To see (2),  note that $\phi_J^{-1}([c,d])$ is a subinterval of $[0,1]$ with length  $\geq \rho^k$. By (1), there exists $J'\in \{0,1,\ldots, m\}^k$ such that
$\phi_{J'}(0)\in \phi_J^{-1}([c,d])$. Therefore $\phi_{JJ'}(0)\in [c,d]$.
\end{proof}

\begin{proof}[Proof of Theorem \ref{thm-2.1}] We divide the proof into some small steps.

{\it Step 1}. Let $0<\delta<1$. We claim that there is a finite set $\Gamma_\delta\subset [0,1-\delta]$ such that for each $n\in \N$ and $I,J\in \{0,1,\ldots,m\}^n$,
\begin{equation}\label{e-3}
\mbox{either}\qquad \rho^{-n}|\phi_I(0)-\phi_{J}(0)|>1-\delta\qquad
\mbox{or} \qquad \rho^{-n}|\phi_I(0)-\phi_{J}(0)|\in \Gamma_\delta.
\end{equation}

To prove  the above claim, we use an idea in \cite{FeLa09}.  Since $\Phi$ satisfies the weak separation condition,  according to  the pigeon-hole principle,  we have
\begin{equation}
\label{e-4}
\sup_{(x,k):\; x\in [0,1],\; k\in \N}\# \left\{\phi_I(0): \; \phi_I(0)\in [x, x+\rho^k],\; I\in \{0,1,\ldots, m\}^k\right\}:=\ell<\infty,
\end{equation}
where $\# X$ denotes the cardinality of $X$.  Indeed, we have $\ell\leq1/c+1$, where $c$ is the constant in Definition \ref{de-1}.

Pick $x\in [0,1]$ and $k\in \N$ so that  the supremum in \eqref{e-4} is attained at $(x,k)$.   Clearly, the supremum in \eqref{e-4} is also attained at $(\phi_I(x),n+k)$ for any $n\in \N$ and $I\in \{0,1,\ldots, m\}^n$.  Pick a large integer $k'$ so that $\rho^{k'}+\rho^{k'+k}<1$ and let
$$x_0=\phi_{0^{k'}}(x),\quad k_0=k'+k.$$ Then $[x_0, x_0+\rho^{k_0}]\subset [0,1]$ and the supremum in \eqref{e-4} is attained at $(x_0,k_0)$.  Choose $W_1,\ldots, W_\ell\in \{0,1,\ldots,m\}^{k_0}$ such that $\phi_{W_1}(0),\ldots, \phi_{W_\ell}(0)$ are different points in $[x_0, x_0+\rho^
{k_0}]$.

Fix $0<\delta< 1$. Pick $k_1\in \N$ so that
\begin{equation}
\label{e-5}
\rho\delta\leq \rho^{k_1}<\delta.
\end{equation}

Now suppose that $I,J\in \{0,1,\ldots,m\}^n$ for some $n\in \N$ such that $$|\phi_I(0)-\phi_J(0)|\leq (1-\delta)\rho^n.$$ Without loss of generality, assume that
$\phi_I(0)\leq \phi_J(0)$. Denote $\Delta=[\phi_J(0), \phi_I(0)+\rho^n]$. Clearly $\Delta\subset \phi_I([0,1])\cap \phi_J([0,1])$, and $|\Delta|\geq \delta \rho^n$, where $|\Delta|$ denotes the length of $\Delta$. Since $\phi_I(0)+\rho^n=\phi_I(1)$, we see that $\phi_I^{-1}(\Delta)=[u,1]$ for some $u\in (0,1)$ with $1-u\geq \delta>\rho^{k_1}$. Set  $I'=\underbrace{m\ldots m}_{k_1}$.  Since $\phi_m(1)=1$, we have $\phi_{I'}(1)=1$.  Observe that $\phi_{I'}([0,1])$ has length  $\rho^{k_1}$,  therefore  $\phi_{I'}([0,1])=[1-\rho^{k_1},1]\subset [u,1]$,
and thus $\phi_{II'}([0,1])\subset \phi_I([u,1])=\Delta$; in particular,
$$
\phi_{II'}([x_0,x_0+\rho^{k_0}])\subset \Delta\subset \phi_J([0,1]).
$$
Note that  $\phi_{II'}([x_0,x_0+\rho^{k_0}])$ is a subinterval of $\phi_J([0,1])$ with length $\rho^{n+k_0+k_1}$. By Lemma \ref{lem-2.2}(2), there exists $J'\in \{0,1,\ldots,m\}^{k_0+k_1}$ such that
$\phi_{JJ'}(0)\in \phi_{II'}([x_0,x_0+\rho^{k_0}])$. Let $x_1=\phi_{II'}(x_0)$. Then  $\phi_{II'}([x_0,x_0+\rho^{k_0}])=[x_1, x_1+\rho^{n+k_0+k_1}]$. Recall that $\phi_{W_1}(0),\ldots, \phi_{W_\ell}(0)$ are different points in $[x_0, x_0+\rho^
{k_0}]$, hence
$\phi_{II'W_1}(0),\ldots, \phi_{II'W_\ell}(0)$ are $\ell$ distinct points in $[ x_1, x_1+\rho^{n+k_0+k_1}]$. Since $\phi_{JJ'}(0)\in [x_1,  x_1+\rho^{n+k_0+k_1}]$, by the maximality of $\ell$ (cf. \eqref{e-4}), we must have
$$
\phi_{JJ'}(0)\in \left\{\phi_{II'W_j}(0):\; 1\leq j\leq \ell \right\}.
$$
That is,
$$
\phi_{J}(0)+\rho^n \phi_{J'}(0)\in \left\{\phi_{I}(0)+\rho^n \phi_{I'W_j}(0):\; 1\leq j\leq \ell \right\}.
$$
It follows that
\begin{eqnarray*}
\rho^{-n}(\phi_{J}(0)-\phi_I(0))&\in & \left\{\phi_{I'W_j}(0)-\phi_{J'}(0):\; 1\leq j\leq \ell \right\}\\
&\subset & \left\{
\phi_{\tilde{I}}(0)-\phi_{\tilde{J}}(0):\;  \tilde{I},\tilde{J}\in \{0,1,\ldots,m\}^{k_0+k_1} \right\}.
\end{eqnarray*}
Hence we can finish the proof of the claim  in Step 1 by setting
\begin{equation}
\label{e-6}
\Gamma_\delta=\left\{
\phi_{\tilde{I}}(0)-\phi_{\tilde{J}}(0):\;  \tilde{I},\tilde{J}\in \{0,1,\ldots,m\}^{k_0+k_1} \right\}\cap [0,1-\delta].
\end{equation}

 \medskip

 {\it Step 2}. Denote $\gamma=\min\{ b_1, b_m-b_{m-1}\}$ and  $B=\{b_i-b_j:\; 0 \leq i,j\leq m\}$. By \eqref{e-1'} and \eqref{e-2.1},  $0<\gamma\leq \rho<1$. Let $\Gamma_\gamma$ be given as in Step 1 (in which we take $\delta=\gamma$). Set
 $$\eta:=\max \left({\rho}^{-1}(\pm \Gamma_\gamma+ B)\cap [0,1)\right).$$
 Clearly $0\leq\eta<1$. We claim that for any $n\in \N$ and $I,J\in \{0,1,\ldots,m\}^n$,
 \begin{equation}\label{e-7}
\mbox{either}\qquad \rho^{-n}|\phi_I(0)-\phi_{J}(0)|\geq 1\qquad
\mbox{or} \qquad \rho^{-n}|\phi_I(0)-\phi_{J}(0)|\leq \eta.
\end{equation}

Assume the claim is not true. Then we can find $n\in \N$ and $I,J\in  \{0,1,\ldots,m\}^n$,
such that
\begin{equation}
\label{e-7'}
\eta<\rho^{-n}(\phi_J(0)-\phi_{I}(0))<1.
\end{equation}
Assume further that the above $n$ is the smallest. As below we derive a contradiction.

First we show that $n\geq 2$. For otherwise, $n=1$ and by \eqref{e-7'}, $0<\rho^{-1}(\phi_J(0)-\phi_I(0))<1$, and  hence  $\rho^{-1}(\phi_J(0)-\phi_I(0))\in {\rho}^{-1}B\cap [0,1)$; by the definition of $\eta$  and the fact $0\in \Gamma_\gamma$,  we have $\rho^{-1}(\phi_J(0)-\phi_I(0))\leq \eta$, a contraction to \eqref{e-7'}.

Since $n\geq 2$, we can write
$$
I=I'i, \quad J=J'j,
$$
where $I',J'\in \{0,1,\ldots,m\}^{n-1}$ and $i,j\in \{0,1,\ldots,m\}$. Then we have
$$
\phi_I(0)=\phi_{I'}(0)+\rho^{n-1}b_i,\quad \phi_{J}(0)=\phi_{J'}(0)+\rho^{n-1}b_j.
$$
Therefore,
\begin{equation}
\label{e-8}
\phi_{J'}(0)-\phi_{I'}(0)=\phi_{J}(0)-\phi_{I}(0)+\rho^{n-1}(b_i-b_j).
\end{equation}
By \eqref{e-8} and \eqref{e-7'}, we have \begin{equation}
\label{e-8'}
|\phi_{J'}(0)-\phi_{I'}(0)|<\rho^{n}+\rho^{n-1}(1-\rho)=\rho^{n-1}.
\end{equation}
 In the following we show further that
\begin{equation}
\label{e-9}
|\phi_{J'}(0)-\phi_{I'}(0)|\leq (1-\gamma)\rho^{n-1}.
\end{equation}

By \eqref{e-8} and the fact that $\phi_{J}(0)>\phi_I(0)$, we have
\begin{eqnarray}
\label{e-10}
\phi_{J'}(0)-\phi_{I'}(0)>\rho^{n-1}(b_i-b_j)\geq -\rho^{n-1}(1-\rho)\geq -\rho^{n-1}(1-\gamma).
\end{eqnarray}
To get an upper bound for $\phi_{J'}(0)-\phi_{I'}(0)$, we consider the following two scenarios respectively:
\begin{itemize}
\item[(i)] $(i,j)=(m,0)$;
\item[(ii)] $(i,j)\neq (m,0)$.
\end{itemize}
First assume that (i) occurs. Then by \eqref{e-8},
\begin{eqnarray*}
\phi_{J'}(0)-\phi_{I'}(0)=\phi_{J}(0)-\phi_{I}(0)+\rho^{n-1}(1-\rho),
\end{eqnarray*}
from which and \eqref{e-7'} we obtain
\begin{eqnarray*}
\frac{\phi_{J'}(0)-\phi_{I'}(0)}{\rho^{n-1}}&=&\frac{\phi_{J}(0)-\phi_{I}(0)}{\rho^{n-1}}+(1-\rho)\\
&=&\frac{\phi_{J}(0)-\phi_{I}(0)}{\rho^n}+(1-\rho)\left(1-\frac{\phi_J(0)-\phi_I(0)}{\rho^n}\right)\\
&>&\frac{\phi_{J}(0)-\phi_{I}(0)}{\rho^n}>\eta.
\end{eqnarray*}
This together with \eqref{e-8'} yields that $1>\rho^{-(n-1)}(\phi_{J'}(0)-\phi_{I'}(0))>\eta$, contradicting the minimality of $n$. Hence (i) can not happen, and (ii) must occur.  Since $(i,j)\neq (m,0)$, we have
$$b_j-b_i\geq \min\{b_1-b_m, \; b_0-b_{m-1}\}=\min\{b_1-(1-\rho), \; -b_{m-1}\}.$$
This together with \eqref{e-8} yields
\begin{equation}
\label{e-11}
\begin{split}
\phi_{J'}(0)-\phi_{I'}(0)&\leq \rho^n-\rho^{n-1}\cdot \min\{b_1-(1-\rho), \; -b_{m-1}\}\\
&=\rho^{n-1}\cdot\max\{ 1-b_1, 1-(b_m-b_{m-1})\}\\
&=\rho^{n-1}(1-\gamma).
\end{split}
\end{equation}
Now \eqref{e-9} follows from \eqref{e-10} and \eqref{e-11}.

According to \eqref{e-9} and the claim in Step 1, we have $\rho^{-(n-1)}|\phi_{J'}(0)-\phi_{I'}(0)|\in \Gamma_\gamma$.
Then by \eqref{e-8},
\begin{eqnarray*}
\rho^{-n}(\phi_{J}(0)-\phi_{I}(0))&=&\rho^{-n}(\phi_{J'}(0)-\phi_{I'}(0))+\rho^{-1}(b_j-b_i)\\
&\in& \rho^{-1} (\pm \Gamma_\gamma+B).
\end{eqnarray*}
This together with \eqref{e-7'} yields $\rho^{-n}(\phi_{J}(0)-\phi_{I}(0))\in \rho^{-1} (\pm \Gamma_\gamma+B)\cap [0,1)$. By the definition of $\eta$, we have $\rho^{-n}(\phi_{J}(0)-\phi_{I}(0))\leq \eta$, which contradicts \eqref{e-7'}. This proves  \eqref{e-7}.

\medskip
{\it Step 3}. Let $\eta\in [0,1)$ be defined as in Step 2. Combining  \eqref{e-7} with the claim in Step 1, we have
for any $n\in \N$ and $I,J\in \{0,1,\ldots,m\}^n$,
 \begin{equation*}
\mbox{either}\qquad \rho^{-n}|\phi_I(0)-\phi_{J}(0)|\geq 1\qquad
\mbox{or} \qquad \rho^{-n}|\phi_I(0)-\phi_{J}(0)|\in \Gamma_{1-\eta},
\end{equation*}
where $\Gamma_1:=\{0\}$.
Hence $\Phi$ satisfies the finite type condition. This finishes the proof of Theorem \ref{thm-2.1}.
\end{proof}

\section{Final remarks and  open questions}
\label{S-3}
\subsection{}
It is worth mentioning a connection between the topological property of $Y_m(q)$ and  the following famous unsolved question: suppose $q>1$ is such that $\|\lambda q^n\|\to 0$ as $n\to \infty$ for some real number $\lambda> 0$, can we assert that
$q$ is a Pisot number? here $\|x\|$ denotes the absolute value of the difference between $x$ and the nearest integer.  It was answered positively by Pisot \cite{Pis38} (see also \cite{Sal63}) if one of the following conditions is satisfied in addition: (i) $\|\lambda q^n\|$ tends to $0$ rapidly enough so that $\sum_{n=1}^\infty\|\lambda q^n\|^2<\infty$, or  (ii) $q$ is  an algebraic number.

We remark that  Theorem \ref{thm-1.1} (also Bugeaud's result in \cite{Bug96})
 implies the following weaker result: \begin{equation}
 \label{e-imp}
 \sum_{n=1}^\infty\|\lambda q^n\|<\infty\Longrightarrow q \mbox{  is a Pisot number}.
 \end{equation}
To see it, assume that $\sum_{n=1}^\infty\|\lambda q^n\|<\infty$. Pick a positive integer $m>q-1$. Take a large integer $N$ so that $\sum_{n\geq N}\|\lambda q^n\|<1/(3m)$.  Then $\|y\|<1/3$ for any real number $y$ in the set $F=\left\{\sum_{i=N}^{n+N}\epsilon_i \lambda q^i:\; \epsilon_i\in \{0, \pm 1,\ldots, \pm m\},\; n=0, 1,\ldots \right\}.$
Hence $F$ is not dense in $\R$. Note that  $Y_m(q)=F/(\lambda q^N)$.   So $Y_m(q)$ is not dense in $\R$. Therefore by Theorem \ref{thm-1.1}, $q$ is a Pisot number. 

As pointed out  by an anonymous referee, using Theorem \ref{thm-1.1},  the  implication \eqref{e-imp}  also follows from  the following inequality 
$$\ell_1(q)\geq (\lambda q^{N})^{-1} \Big(1-\sum_{n=N}^\infty\|\lambda q^n\|\Big) \quad \mbox{ if } \quad \sum_{n=N}^\infty\|\lambda q^n\|< \frac1{q+1}.
$$
This inequality  is only formulated in  \cite[Theorem 1]{EJK98} in the case when $\lambda=1$, but it extends to $\lambda>0$ with the identical proof.   

\subsection{} We remark that the proof of Theorem 1.9 implies the following result, which is of
interest in its own right.

\begin{pro}
Under the assumptions of Theorem \ref{thm-2.1}, there exists $k\in \N$ such that for any $n\in \N$, $I, J\in \{0,1,\ldots,m\}^n$,
if $\rho^{-n}|\phi_I(0)-\phi_J(0)|<1$, then there exist  $I', J'\in \{0,1,\ldots,m\}^k$ such that $\phi_{II'}(0)=\phi_{JJ'}(0)$.
\end{pro}

As a corollary, we have

\begin{cor}
Assume that $m\in \N$ and q is a Pisot number in $(1,m+1]$. Then there exists $k\in \N$ so that if
$
\left|\sum_{i=0}^{n-1} \epsilon_i q^i\right|< \frac{m}{q-1}
$
for some $n\in \N$ and  $\epsilon_0,\ldots, \epsilon_{n-1}\in \{0,\pm 1,\ldots, \pm m\}$, then there exist
$\epsilon_{n},\ldots, \epsilon_{n+k-1}\in \{0,\pm 1,\ldots, \pm m\}$ such that
$$\sum_{i=0}^{n+k-1}\epsilon_i q^i=0.$$
\end{cor}

Similar to Pisot numbers, there is certain separation property about  Salem numbers. Recall that a  number $q>1$ is called a {\it Salem number}
if it is an algebraic integer whose algebraic conjugates all have modulus no
greater than 1, with at least one of which on the unit circle. It follows from Lemma 1.51 in  Garsia \cite{Gar62} that  if $q$ is a Salem number and $m\in \N$, then
there exist $c>0$ and $k\in \N$ ($c, k$ depend on $q$ and $m$) such that
\begin{equation}
\label{e-end}
Y_m^n(q) \cap \left(-cn^{-k}, cn^{-k}\right)=\{0\},\quad \forall \;n\in \N,
\end{equation}
where $Y_m^n(q):=\left\{\sum_{i=0}^{n-1}\epsilon_i q^i:\; \epsilon_i\in \{0, \pm 1,\ldots, \pm m\}\right\}$.  We end the paper by posing the following  questions.

\begin{itemize}
\item For $m\in \N$ and a non-Pisot number $q\in (1, m+1)$, does the property \eqref{e-end} imply that $q$ must be a Salem number?
\item Does Theorem \ref{thm-2.1} still hold without the assumption \eqref{e-2.1}?
\end{itemize}

\noindent{\bf Acknowledgements}.  The research  was supported by  RGC grants in the Hong Kong Special Administrative Region, China (projects CUHK401112, CUHK401013). The author thanks Shigeki Akiyama for sending him a copy of \cite{AkKo11}. He is grateful to Nikita Sidorov for pointing out the implication \eqref{e-imp} in section 3.1, and to Shigeki Akiyama, Vilmos Komornik and  Nikita Sidorov for pointing out the fact that Theorem \ref{thm-1.1} implies Theorem \ref{thm-1.3'}. He also  thanks the anonymous
referees and Toufik Zaimi for the very careful reading of the original manuscript and the many valuable
comments to improve the paper.

\end{document}